\documentclass{amsart}[]

\usepackage[latin1]{inputenc}
\usepackage{amsfonts}
\usepackage{amsmath}
\usepackage{amsthm}
\usepackage{amssymb}
\usepackage{latexsym}
\usepackage{enumerate}
\usepackage{comment} 
\usepackage{hyperref}
\usepackage[all]{xy}
\usepackage{xcolor}
\newtheorem{theorem}{Theorem}[section]
\newtheorem{lemma}[theorem]{Lemma}
\newtheorem{proposition}[theorem]{Proposition}
\newtheorem{corollary}[theorem]{Corollary}

\newtheorem{remark}[theorem]{Remark}

\newtheorem{definition}[theorem]{Definition}

\numberwithin{equation}{section}

\begin{document}

\newcommand{\QQ}{\mathbb{Q}}
\newcommand{\NN}{\mathbb{N}}
\newcommand{\GG}{\mathbb{G}}
\newcommand{\cc}{\mathfrak{c}}
\newcommand{\N}{\mathbb{N}}
\newcommand{\Q}{\mathbb{Q}}
\newcommand{\C}{\mathbb{C}}
\newcommand{\Z}{\mathbb{Z}}
\newcommand{\R}{\mathbb{R}}
\newcommand{\T}{\mathbb{T}}
\newcommand{\I}{\mathcal{I}}
\newcommand{\J}{\mathcal{J}}
\newcommand{\A}{\mathcal{A}}
\newcommand{\HH}{\mathcal{H}}
\newcommand{\K}{\mathcal{K}}
\newcommand{\B}{\mathcal{B}}
\newcommand{\st}{*}
\newcommand{\PP}{\mathbb{P}}
\newcommand{\SSS}{\mathbb{S}}
\newcommand{\forces}{\Vdash}
\newcommand{\dom}{\text{dom}}
\newcommand{\osc}{\text{osc}}
\newcommand{\irr}[1]{irr({#1})}
\newcommand\encircle[1]{%
  \tikz[baseline=(X.base)] 
    \node (X) [draw, shape=circle, inner sep=0] {\strut #1};}

\title{Uncountable almost Irredundant Sets in nonseparable C*-algebras}
\author{Clayton Suguio Hida}

\begin{abstract} In this article, we consider the notion of almost irredundant sets: A subset $\mathcal{X}$ of a C*-algebra $\A$ is called almost irredundant if and only if for every $a\in \mathcal{X}$, the element $a$ does not belong to the norm-closure of 
$$\{\sum_{i=1}^n \lambda_i \prod_{j=1}^{n_i}a_{{i,j}}: \textrm{ where } a_{{i,j}} \in \mathcal{X}\setminus\{a\} \textrm{ and} \sum |\lambda_i|\leq 1\}.$$
Since every almost irrredundant set is in particular a discrete set, it follows that the density of $\A$ is an upper bound for  the size of almost irredundant sets.  We prove that under the Proper Forcing Axiom (PFA), there is an uncountable almost irredundant set in every C*-algebra with an uncountable increasing sequence of ideals. In particular, assuming PFA, every nonseparable scattered C*-algebra admits an uncountable almost irredundant set.
 


\end{abstract}

\maketitle
\tableofcontents

\section{Introduction}
The notion of irredundant sets in C*-algebras was introduced in \cite{hida2018large}:
\begin{definition}Let $\mathcal{A}$ be a C*-algebra. A subset $\mathcal X\subseteq \mathcal{A}$
 is called irredundant  if and only if for every $a \in \mathcal X$, the 
 C*-subalgebra of $\mathcal{A}$ generated by $\mathcal X\setminus\{a\}$ 
 does not contain $a$.
\end{definition}
Since every irredundat set is in particular a discrete set, it follows that separable C*-algebras admits only countable irredundant sets and it is natural to ask whether every nonseparable C*-algebra admits an uncountable irredundant set. For commutative C*-algebras, the classical example of $C(K)$ where $K$ is the Kunen space is a consistent example of a nonseparable commutative C*-algebra without uncountable irredundant set.

In \cite{hida2018large}, the authors asked if the noncommutativity could take an important role in the existence of an uncountable irredundant set. Assuming the Diamond Principle $\Diamond$, they constructed a nonseparable fully-noncommutative scattered\footnote{See \cite{jensen1978scattered, jensen1979scattered, wojtaszczyk1974linear, lin1989structure, kusuda2012c, tomiyama1963characterization, ghasemi2018noncommutative} for the definition and main results about scattered C*-algebras.} C*-algebra without uncountable irredundant set (see Theorem 6.2 of \cite{hida2018large}).

On the other hand, we do not know if it is consistent that every nonseparable C*-algebra admits an uncountable irredundant set. Even in the commutative case, the question remains open. But, if we restrict the question to the class of commutative scattered C*-algebras, then we have the following\footnote{Actually, the result holds for every 0-dimensional compact spaces.}:

\begin{theorem}[S.Todorcevic \cite{irrstevo93,quocientstevo}]Assume Martin's Axiom. Then every nonseparable commutative C*-algebra of the form $C(K)$ for $K$ a scattered compact space, admits an uncountable irredundant set. 
\end{theorem}
 
The main idea of this article was to generalize this result for every scattered C*-algebra. We were able to prove a weaker version, replacing the notion of irredundant sets by a weaker version:

\index{almost irredundant sets}
\begin{definition}\label{almost}Let $\A$ be a C*-algebra. A subset $\mathcal{X}$ of a C*-algebra $\A$ is called almost irredundant if and only if for every $a\in \mathcal{X}$, the element $a$ does not belong to the norm-closure of 
$$\{\sum_{i=1}^n \lambda_i \prod_{j=1}^{n_i}a_{{i,j}}: \textrm{ where } a_{{i,j}} \in \mathcal{X}\setminus\{a\} \textrm{ and} \sum |\lambda_i|\leq 1\}.$$
\end{definition}

Clearly, every irredundant set is almost irredundant, but this two notions are not equivalent. For instance, consider $\mathcal{A}=M_2(\mathbb{C})$. Let $p_1,p_2$ be projections onto $(1,0)$ and onto $(0,1)$ respectively, and $I_d$ the identity of $\mathcal{A}$. Then $\{p_1, p_2, Id\}$ is almost irredundant but it is not irredundant.

The notion of almost irredundant set can be thought as a stronger version of the notion of $\kappa$ - polyhedron: A bounded family $(x_\alpha)_{\alpha<\kappa}$ in a Banach space $X$ is a $\kappa$- polyhedron if $x_\alpha \notin \overline{cov}(x_\beta, \beta<\kappa, \beta \neq \alpha)$ for every $\alpha<\kappa$.
Recall that a Banach space X has the Kunen - Shelah property $KS_4$ if $X$ does not have an $\omega_1$- polyhedron \footnote{See \cite{granero2003kunen} for details.}. 

In particular, if a C*-algebra $\A$ has an uncountable almost irredundant set, then $\A$ has an $\omega_1$- polyhedron and therefore, does not have the Kunen-Shelah property $KS_4$.

By Proposition 2.2 of \cite{granero2003kunen}, the existence of an $\omega_1$- polyhedron is equivalent to the existence of an uncountable bounded almost biorthogonal system (UBABS) of type $\eta$ for some $0\leq \eta<1$:
\begin{definition} $ $
\begin{itemize}
\item An uncountable bounded almost biorthogonal system (UBABS) of type $\eta$ in a Banach space $X$ is a bounded uncountable sequence $(x_\alpha, x_\alpha^*)_{\alpha<\omega_1} \in X\times X^*$ such that $x_\alpha^*(x_\alpha) = 1$ and $|x_\alpha^*(x_\beta)|\leq\eta$ for all $\alpha\neq \beta$ (see Section 2 of \cite{granero2003kunen}). 

\item In the case where $\A$ is a C*-algebra and $(a_\alpha, \tau_\alpha)_{\alpha<\omega_1}$ is a UBABS of type $\eta$ such that $\tau_\alpha$ is a positive linear functional of norm one (i.e., a state) and $0<a_\alpha\leq 1$ is a positive element for every $\alpha<\omega_1$, we say that $(a_\alpha, \tau_\alpha)_{\alpha<\omega_1}$ is an uncountable positive bounded almost biorthogonal system (UBABS+). 
\end{itemize}
\end{definition}

For almost irredundant sets, we prove the following:

\begin{theorem}\label{discretalmost_intro}
Let $\A$ be an unital C*-algebra. Suppose $(a_\alpha,\tau_\alpha)_{\alpha<\omega_1}$ is a UBABS+ of type $\eta$ for $0\leq \eta<1$.
Then $\mathcal{X} = \{a_\alpha: \alpha<\omega_1\}$ is an uncountable almost irredundant set.
\end{theorem}

For commutative C*-algebras, the existence of an uncountable discrete set of pure states guarantee the existence of an UBABS+ of type $\eta=0$. This follows from the Urysohn's Lemma. In fact, suppose $K$ is a compact Hausdorff space and $(x_\alpha)_{\alpha<\omega_1}$ is a discrete set in $K$. For each $\alpha<\omega_1$, consider $f_\alpha:K\to [0,1]$ in $C(K)$ such that $f_\alpha(x_\alpha)=1$ and $f_\alpha(x_\beta)=0$ for every $\beta\neq \alpha$. Then $(f_\alpha,\delta_{x_\alpha})_{\alpha<\omega_1}$ is a UBABS+ of type $\eta=0$.

In this direction, we obtain the following:

\begin{proposition}\label{discrete-positive-intro}Let $\A$ be an unital C*-algebra. Suppose $(\tau_\alpha)_{\alpha<\omega_1}$ is an uncountable discrete set in $S(\A)$ such that $\tau_\alpha$ can be excised\footnote{See Definition \ref{excision}.} for every $\alpha<\omega_1$. Then there are an uncountable subfamily $(\beta_\alpha)_{\alpha<\omega_1}$  of $\omega_1$, an uncountable family  $(a_\alpha)_{\alpha<\omega_1}$ of $\A_{+,1}$ and $0\leq \eta<1 $ such that $(a_\alpha,\tau_{\beta_\alpha})_{\alpha<\omega_1}$ is a UBABS+ of type  $\eta$.
In particular, by Theorem \ref{discretalmost_intro}, $\mathcal{X} = \{a_\alpha: \alpha<\omega_1\}$ is an uncountable almost irredundant set. In the case where $\A$ is of real rank zero\footnote{We say that a C*-algebra $\A$ has real rank zero if the set of all self-adjoint elements with finite spectrum is dense in $\A_{sa}$.}, we can assume that $(a_\alpha)_{\alpha<\omega_1}$ is a family of projections.
\end{proposition}

In particular, if $\A$ is a C*-algebra such that there is an uncountable discrete set of pure states, then there is an UBABS+ of type $\eta$ for some $0\leq \eta <1$. This follows from Proposition \ref{discrete-positive-intro} and by the fact that every pure state can be excised (Proposition 2.2 of \cite{akemann1986excising}).

In the search of discrete sets of pure states, we obtained the following result:

\begin{theorem}
Let $\mathcal{A}$ be an unital C*-algebra and suppose that $\A$ has uncountable positive semibiorthogonal system\footnote{See Definition \ref{semibiorthogonal}.} of pure states. Then
\begin{enumerate}
\item Assuming Martin's Axiom, the state space $S(\A)$ has an uncountable discrete set.
\item If $\A$ is a commutative C*-algebra, then assuming Martin's Axiom, $\A$ has an UBABS+ (and therefore, an uncountable almost irredundant set).
\item Assuming the Proper Forcing Axiom, the space of pure states  $P(\A)$, has an uncountable discrete set. 
\end{enumerate}
\end{theorem}
\begin{proof}
See Theorem \ref{mateorema} and Theorem \ref{pfateorema}.
\end{proof}

In the last section, we observe that the existence of an uncountable sequence of increasing ideals is sufficient to guarantee the existence of an uncountable positive semibiorthogonal system of pure states. In particular, every C*-algebra with a composition series of length at least $\omega_1$. We conclude the last section with the proof of our main result:

\begin{theorem}\label{main_result_intro}Assuming PFA, every nonseparable scattered C*-algebra admits an uncountable almost irredundant set of projections.
\end{theorem}

Theorem \ref{main_result_intro} complements the construction in  Theorem 6.2 of \cite{hida2018large}, which can be proved that it is a consistent example of a nonseparable scattered C*-algebra without an uncountable almost irredundant set of projections.


Notation and terminology should be standard: Let $\A$ be a C-algebra. The set of all self-adjoint elements of $\A$ will be denoted by $\A_{sa}$ and the set of all positive elements will be denoted by $\A_+$, and $\A_{+,1}$ will be the set of positive elements of norm one.

We denote by $S(A)$ the set of all states of $\A$ equipped with the weak*-topology. By $P(\A)$ we denote the extreme points of $S(A)$, the set of all pure states of $\A$. See  \cite{kenneth1996c}, \cite{dixmier1977c}, \cite{li1992introduction},  \cite{murphy1990c} and \cite{pedersen1979c} for other definitions and results about C*-algebras.

\section{Almost irredundant sets and UBABS+}
In this section we study the relation between irredundance and UBABS+. We will prove that the existence of an UBABS+ of type $\eta$ for some $0\leq \eta< 1$ is enough to guarantee the existence of an uncountable almost irredundant set.

We start by proving the following simple inequality about states and positive elements:

\begin{lemma}\label{inequalityalmost}Let $\A$ be an unital C*-algebra. Suppose $a_1, \cdots, a_n$ are self adjoint elements such that $0\leq a_i\leq 1$ and $\tau(a_i)<\varepsilon$ for $\tau$ a state. Then 
$$\tau(a_1 \cdots a_n)\leq \sqrt{\varepsilon}.$$

\end{lemma}
\begin{proof}$ $
The result follows from repeatedly applying the inequalities  $|\tau(a)|^2\leq \|\tau\|\tau(a^*a)$ and ${\tau(b^*a^*ab)\leq \|a^*a\|\tau(b^*b)}$, which holds for every positive linear functional $\tau$ and every $a,b \in \A$ (see Theorem 3.3.2 and Theorem 3.3.7 of \cite{murphy1990c}).



\end{proof}

\begin{theorem}\label{discretalmost}

Let $\A$ be an unital C*-algebra. Suppose $(a_\alpha,\tau_\alpha)_{\alpha<\omega_1}$ is an UBABS+ of type $\eta$ for some $0\leq \eta<1$.
Then $\mathcal{X} = \{a_\alpha: \alpha<\omega_1\}$ is an almost irredundant set.
\end{theorem}
\begin{proof}
Suppose $\mathcal{X} = \{a_\alpha: \alpha<\omega_1\}$ is not an almost irredundant set and lets get a contradiction.
Consider $\alpha<\omega_1$ such that $a_\alpha$ belongs to the closure of 
$$\{a=\sum_{i=1}^n \lambda_i \prod_{j=1}^{n_i}a_{\beta_{i,j}}: \textrm{ where } \beta_{i,j}\neq \alpha \textrm{ and} \sum |\lambda_i|\leq1\}.$$

Consider $\varepsilon>0$ such that $\varepsilon + \sqrt{\eta}<1$ and $a=\sum_{i=1}^n \lambda_i \prod_{j=1}^{n_i}a_{\beta_{i,j}}$ such that $\|a_\alpha-a\|<\varepsilon.$ In particular,
$$|\tau_\alpha(a_\alpha-a)|< \varepsilon.$$
On the other hand,

$$|\tau_\alpha(a)| = |\tau_\alpha(\sum_{i=1}^n \lambda_i \prod_{j=1}^{n_i}a_{\beta_{i,j}})|=|\sum_{i=1}^n \lambda_i \tau_\alpha(\prod_{j=1}^{n_i}a_{\beta_{i,j}})|\leq \sum_{i=1}^n |\lambda_i| |\tau_\alpha(\prod_{j=1}^{n_i}a_{\beta_{i,j}})|\leq $$
$$\stackrel{*}{\leq} \sum_{i=1}^n |\lambda_i|\sqrt{\eta}\leq \sqrt{\eta}$$
where in $*$ we applied Lemma \ref{inequalityalmost}.
But then, 
$$1 = |\tau_\alpha(a_\alpha)| = |\tau_\alpha(a_\alpha-a) + \tau_\alpha(a)|\leq |\tau_\alpha(a_\alpha-a)| + |\tau_\alpha(a)| <\varepsilon + \sqrt{\eta} <1$$
which is a contradiction.

This proves that $\mathcal{X} = \{a_\alpha: \alpha<\omega_1\}$ is an almost irredundant set in $\A$. 
\end{proof}

We would like to conclude this section with some results about UBABS+ of type $\eta=0$.
 
If $(a_\alpha,\tau_\alpha)_{\alpha<\omega_1}$ is a UBABS+ of type $\eta=0$, then $(a_\alpha,\tau_\alpha)_{\alpha<\omega_1}$ is a biorthogonal system such that $\tau_\alpha$ is a positive linear functional and $a_\alpha$ is a posive element for every $\alpha<\omega_1$. For simplicity of notation, we simply say that a $(a_\alpha,\tau_\alpha)_{\alpha<\omega_1}$ is a positive biorthogonal system.

By Lemma 3.14 of \cite{hida2018large}, if a C*-algebra has an uncountable positive biorthogonal system, then it contains an uncountable irredundant set. Actually, it can be proved that the existence of a positive biorthogonal system is equivalently with the following stronger version of irredundance:

\begin{definition}Let $\mathcal{A}$ be a C*-algebra. A subset $\mathcal X\subseteq \mathcal{A}_{sa}$
 is called h-irredundant  if and only if for every $a \in \mathcal X$, the 
 hereditary C*-subalgebra of $\mathcal{A}$ generated by $\mathcal X\setminus\{a\}$ 
 does not contain $a$.
\end{definition}

\begin{proposition}\label{states-irredundance}Let $\mathcal{A}$ be C*-algebra and $\mathcal{X} = \{a_\alpha: \alpha<\kappa\}$ a sequence in $\mathcal{A}_{+}$. Then $\mathcal{X}$ is h-irredundant if and only if for each $\alpha<\kappa$, there is a positive functional $\tau_\alpha$ such that 
$(a_\alpha,\tau_\alpha)_{\alpha<\omega_1}$ is a positive biorthogonal system.
\end{proposition}

\begin{proof}
$(\Rightarrow)$ Suppose $(a_\alpha)_{\alpha<\kappa}$ is h-irredundant. For each $\alpha$, consider $H_\alpha$ the hereditary C*-subalgebra generated by $\{a_\beta:\beta<\kappa, \beta\neq \alpha\}$. Let $\mathcal{B}$ be the C*-subalgebra generated by $(a_\alpha)_{\alpha<\kappa}$.

For each $\alpha<\kappa$, as $H_\alpha \varsubsetneq \mathcal{B}$, by Theorem 5.3.1 of \cite{murphy1990c}, there is a non null positive functional $\tau_\alpha'$ on $\mathcal{B}$ vanishing on $H_\alpha$. We have that 
$\tau_\alpha'(a_\alpha)\neq 0$. Otherwise, we would have $a_\alpha$ in the left kernel $N_{\tau_\alpha'} = \{b\in \mathcal{B}: \tau_\alpha'(b^*b)=0\}$ and since $H_\alpha\subset N_{\tau_\alpha'}$ we would get $\mathcal{B}= N_{\tau_\alpha'}$, contradiction with the fact that $\tau_\alpha'$ is non null on $\mathcal{B}$. To conclude, consider $\tau_\alpha$ the unique extension of $\tau_\alpha'$ on $\mathcal{A}$.

$(\Leftarrow)$. Analogous with the proof of Lemma 3.14 of \cite{hida2018large}: If there is a family of positive functionals $(\tau_\alpha)_\alpha$ separating $(a_\alpha)_\alpha$, then for each $\alpha<\kappa$, $a_\alpha \notin N_{\tau_\alpha}$ but $a_\beta \in N_{\tau_\alpha}$ for each $\beta \neq \alpha$. Since left kernels are left ideais and in particular, hereditary C*-subalgebras, we conclude that $(a_\alpha)_{\alpha<\kappa}$ is h-irredundant.
\end{proof}

In the case of scattered C*-algebras, we have the following:

\begin{corollary}If $\A$ is a scattered C*-algebra, then $\A$ has a h-irredundant set of cardinality $\kappa$ if and only if $\A$ has a positive biorthogonal system 
$(a_\alpha,\tau_\alpha)_{\alpha<\kappa}$ such that $\tau_\alpha$ is a pure state for each $\alpha<\kappa$.

\end{corollary}
\begin{proof}
It follows from Proposition \ref{states-irredundance} and by the fact that every positive linear functional on a scattered C*-algebra is atomic (See \cite{jensen1978scattered, jensen1979scattered}. Also Theorem 4.8 of \cite{ghasemi2018noncommutative}).
\end{proof}

\section{Discrete set of states and UBABS+}

In this section we prove that the existence of an uncountable discrete set of pure states is enough to guarantee the existence of a UBABS+ of type $\eta$ for some $0\leq \eta<1$.

Recall the following notion of excision of states:
\begin{definition}\label{excision}Let $\A$ be a C*-algebra and $\tau \in S(\A)$. We say that a net of positive elements $(a_\alpha)_{\alpha<\kappa}$ in $\A_{+,1}$ is a excision of $\tau$ if 
$$ \lim_\alpha \|a_\alpha b a_\alpha -\tau(b)a_\alpha\| = 0$$
for every $b\in \A$.

In this case, we say that $\tau$ can be excised.
\end{definition}

The following result give us some insight into the weak*-neighbourhood of a state that can be excised. (see \cite{farah2019combinatorial}, section 5.2 for details):

\begin{proposition}\label{openneighbourhood}Let $\A$ be a C*-algebra and $\tau\in S(\A)$. Suppose that $\tau$ can be excised. Then the sets 
$$U_{a,\varepsilon}:=\{\sigma: \sigma(a)>1-\varepsilon\},$$
where $\varepsilon>0$ and $a\in \A_{+,1}$ is such that $\tau(a)=1$, form a local neighbourhood basis of $\tau$ in $ S(\A)$.
Moreover, if $\A$ is a unital C*-algebra of real rank zero, then we can take $a$ as projections.
\end{proposition}
\begin{proof}
See Lemma 5.2.2 of \cite{farah2019combinatorial} and Proposition 3.13.16 of \cite{eilers2018c}.
\end{proof}

\begin{proposition}\label{discrete-positive}Let $\A$ be an unital C*-algebra. Suppose $(\tau_\alpha)_{\alpha<\omega_1}$ is an uncountable discrete set in $S(\A)$ such that $\tau_\alpha$ can be excised for every $\alpha<\omega_1$. Then there are an uncountable subfamily $(\beta_\alpha)_{\alpha<\omega_1}$  of $\omega_1$, an uncountable family  $(a_\alpha)_{\alpha<\omega_1}$ of $\A_{+,1}$ and $0\leq \eta <1$ such that $(a_\alpha, \tau_{\beta_\alpha})_{\alpha<\omega_1}$ is a UBABS+ of type $\eta$.
In particular, by Theorem \ref{discretalmost}, $(a_\alpha)_{\alpha<\omega_1}$ is an uncountable almost irredundant set. In the case where $\A$ is of real rank zero, we can assume that $(a_\alpha)_{\alpha<\omega_1}$ is a family of projections.

\end{proposition}
\begin{proof}
For each $\alpha<\omega_1$, let $V_\alpha$ be a weak*-open neighbourhood of $\tau_\alpha$ such that $\tau_\beta \notin V_\alpha$ for every $\beta\neq \alpha$.
Since $\tau_\alpha$ can be excised, by Proposition \ref{openneighbourhood}, there are $a_\alpha \in \A_{+,1}$ and $\varepsilon_\alpha>0$ such that
$$\tau_\alpha \in \{\sigma: \sigma(a_\alpha)>1-\varepsilon_\alpha\}\subset V_\alpha.$$
Consider an uncountable subset $\Gamma \subset \omega_1$ and $\varepsilon>0$ such that $\varepsilon_\alpha = \varepsilon$ for every $\alpha \in \Gamma$. Then for $\alpha,\beta\in \Gamma$ with $\alpha\neq \beta$, as $\tau_\beta \notin V_\alpha$, we have that $\tau_\beta (a_\alpha)\leq 1-\varepsilon$.
This shows that $(a_\alpha,\tau_\alpha)_{\alpha\in \Gamma}$ is a UBABS+ of type $\eta =1-\varepsilon $.

\end{proof}

Since every pure state can be excised (Proposition 2.2 of \cite{akemann1986excising}), it follows from Proposition \ref{discrete-positive} that the existence of an uncountable discrete set of pure states is enough to guarantee the existence of an UBABS+ of type $\eta$, and therefore, an uncountable almost irredundant set.

\section{Semibiorthogonal system and uncountable discrete sets of states}\label{pfasection}
This section deals with positive semibiorthogonal system, discrete sets and UBABS+. We prove that under the Martin's Axiom, every C*-algebra with an uncountable positive semibiorthogonal system admits an uncountable discrete set of states and in the case of commutative C*-algebras, we prove the existence of an UBABS+.
In the noncommutative case, we prove that under the Proper Forcing Axiom, every C*-algebra $\A$ with an uncountable positive semibiorthogonal system of pure states has an uncountable discrete set of pure states.

\begin{definition}\label{semibiorthogonal}Let $\A$ be a C*-algebra. A positive semibiorthogonal system on $\A$ is a sequence $(a_\alpha, \tau_\alpha)_{\alpha<\kappa} \in \A\times \A^*$ such that 
\begin{enumerate}
\item $\tau_\alpha(x_\alpha) = 1$
\item $\tau_\alpha(a_\beta)= 0$ for $\beta<\alpha$,
\item $\tau_\alpha(a_\beta)\geq 0$ for all $\alpha\neq \beta$ and
\item $\tau_\alpha$ is a positive linear functional
\end{enumerate}
for every $\alpha, \beta <\kappa$. In the case where $\tau_\alpha$ is a pure state for each $\alpha<\kappa$, we say that $(a_\alpha, \tau_\alpha)_{\alpha<\kappa}$ is a positive semibiorthogonal system of pure states.
\end{definition}

The following partial order will be used to apply the Martin's Axiom:

\begin{definition}\label{def_part}Let $\A$ be a C*-algebra. Consider $(a_\alpha, \tau_\alpha)_{\alpha<\kappa} \in \A\times \A^*$ a positive semibiorthogonal system in $\A$.

Define the partial order $\mathbb{P}$ formed by elements $p\in [\omega_1]^{<\omega}$ and where the order is given by $p\leq q$: 
\begin{enumerate}
\item $q\subseteq p$;
\item For all $\beta \in q$ and for all $\alpha \in p\setminus q$ we have that $\tau_\alpha(a_\beta) <\frac{1}{2}$.
\end{enumerate}
\end{definition}

\begin{lemma}\label{non_ccc}Let $\A$ be a C*-algebra. Consider $(a_\alpha, \tau_\alpha)_{\alpha<\kappa} \in \A\times \A^*$ a positive semibiorthogonal system in $\A$ and $\mathbb{P}$ the partial order of Definition \ref{def_part}. Suppose $\mathbb{P}$ does not have the countable chain condition (c.c.c). Then 

\begin{enumerate}
\item There is an uncountable discrete set in $S(\A)$.
\item If $\A = C(K)$ and $\tau_\alpha$ is a pure state for each $\alpha<\omega_1$, then $\A$ has an UBABS+.
\end{enumerate}

\end{lemma}

\begin{proof}
Let $(p_\xi)_{\xi<\omega_1}$ be an uncountable antichain in $\mathbb{P}$. Suppose $(p_\xi)_{\xi<\omega_1}$ is a $\Delta$-system with root $\Delta$ and such that:
\begin{itemize}
\item $\Delta<p_\xi\setminus \Delta < p_\eta\setminus \Delta$ for each $\xi<\eta<\omega_1$;
\item $|p_\xi\setminus \Delta|=n$,
\item $p_\xi\setminus \Delta = (\alpha_1(\xi), \cdots, \alpha_n(\xi))$.
\end{itemize}
For each $\xi<\eta$ define $p_{\xi,\eta} = p_{\xi}\cup p_{\eta}$. Since $(p_\xi)_{\xi<\omega_1}$ is an antichain, for each $\xi<\eta$, $p_{\xi,\eta}\not \leq p_\xi,p_\eta$.
Since $p_{\xi,\eta}\leq p_\xi$, it follows that $p_{\xi,\eta}\not \leq p_\eta$. In particular, there is $\beta \in p_\eta$ and $\alpha \in p_{\xi, \eta}\setminus p_\eta = p_{\xi}\setminus \Delta$ such that $\tau_{\alpha}(a_\beta)\geq \frac{1}{2}$ (observe that we should have $\beta\in p_\eta\setminus \Delta.$)
\begin{enumerate}
\item 
Consider $F_\eta = \frac{1}{n}\sum_{\alpha \in p_\eta\setminus \Delta} \tau_\alpha$ and $b_\eta = \sum_{\alpha \in p_\eta\setminus \Delta}a_\alpha$ and
$$W_\eta = \{\tau: \tau(b_\eta)\geq \frac{1}{2n}\}$$
$$W_{\eta}^0 = \{\tau: \tau(b_\eta)> \frac{1}{4n}\}$$
Then $W_\eta\subseteq W_\eta^0$, with $W_\eta$ closed and $W_\eta^0$ open set in $S(\A)$.

If $\xi\leq \eta$, then $F_\xi(b_\eta)\geq\frac{1}{2n}$, in particular, $F_\xi \in W_{\eta}$.

If $\xi>\eta$, then $F_\xi(b_\eta)=0 <\frac{1}{4n}$, in particular, $F_\xi \notin W_{\eta}^0$.

This shows that $(F_\xi)_{\xi<\omega_1}$ is a free sequence\footnote{See \cite{handbookconjuntos} for definitions and results about cardinal functions.} in $S(\A)$ i.e., 
$$\overline{\{F_\xi: \xi\leq\eta\}}\cap \overline{\{F_\xi: \xi>\eta\}} = \emptyset$$
for each $\alpha<\omega_1$. In particular, a discrete set in $S(A)$ and this proves (1).

\item 
Suppose now that $\A = C(K)$ and $\tau_\alpha = \delta_{x_\alpha}$ is a pure state for each $\alpha<\omega_1$. 
Then for every $\xi<\eta$, there are $i,j =1, \cdots, n$ such that 
$$\tau_{\alpha_j(\xi)} (a_{\alpha_i(\eta)}) = \delta_{x_{\alpha_j(\xi)}} (a_{\alpha_i(\eta)}) = a_{\alpha_i(\eta)}(x_{\alpha_j(\xi)})\geq\frac{1}{2}.$$
Consider for each $\eta<\omega_1$:

$$W_\eta = \{(x_1, \cdots, x_n)\in K^n: \exists i,j(a_{\alpha_i(\eta)}(x_j)\geq \frac{1}{2})\}$$ 
and
$$W_\eta^0 = \{(x_1, \cdots, x_n)\in K^n: \exists i,j(a_{\alpha_i(\eta)}(x_j)> \frac{1}{4})\}$$ 

Then $W_\eta\subseteq W_\eta^0$, with $W_\eta$ closed and $W_\eta^0$ open set in $K^n$.

For each $\alpha<\omega_1$, define $a_\eta = (x_{\alpha_1(\eta)}, \cdots, x_{\alpha_n(\eta)})\in K^n$. Then 

If $\xi\leq\eta$, then there is $i,j = 1, \cdots, n$ such that $a_{\alpha_i(\eta)}(x_{\alpha_j(\xi)})\geq\frac{1}{2}$. In particular, $a_\xi \in W_\eta$;

If $\xi> \eta$, then $a_{\alpha_i(\eta)}(x_{\alpha_j(\xi)})=0<\frac{1}{4}$ for every $i,j = 1, \cdots, n$ and therefore, $a_\xi \not\in W_\eta^0$

This shows that $(a_\eta)_{\eta<\omega_1}$ is a free sequence in $K^n$. In particular, $K^n$ has an uncountable discrete set. By Proposition 3.5 of \cite{somepiotr11}, it follows that $C(K)$ has an UBABS+ of type $1 - \frac{1}{m}$ for some $m\leq n$.

\end{enumerate}

\end{proof}
It would be interesting to get the conclusion 2) of Lemma \ref{non_ccc} for noncommutative C*-algebras. For this, it would be sufficient to prove that every C*-algebra $\A$ with an uncountable discrete set in $P(\A)^n$ for some $n\in \mathbb{N}$,  admits an UBABS+ (i.e., a noncommutative version of Proposition 3.5 of \cite{somepiotr11}). The proof of Proposition 3.5 of \cite{somepiotr11} relies on the Urysohn lemma, and we have not been able to adapt this proof using the noncommutative analogues of the Urysohn Lemma.

When $\mathbb{P}$ is c.c.c, we can apply the Martin's Axiom to get the following result\footnote{The proof follows the line of the Theorem 11 \cite{quocientstevo}.}:
\begin{theorem}\label{mateorema}Assume MA.
Let $\A$ be a C*-algebra and suppose $\A$ has an uncountable positive semibiorthogonal system of pure states. Then $S(A)$ has an uncountable discrete set. In case where $\A=C(K)$ and $\tau_\alpha$ is a pure state for each $\alpha<\omega_1$, then $\A$ has an UBABS+.

\end{theorem}
\begin{proof}
Let $(a_\alpha, \tau_\alpha)_{\alpha<\kappa} \in \A\times \A^*$ be an uncountable positive semibiorthogonal system of pure states.
Consider $\mathbb{P}$ the partial order of Definition \ref{def_part}. 
If $\mathbb{P}$ does not have the c.c.c, then by Lemma \ref{non_ccc}, $S(\A)$ has an uncountable discrete set (and $\A$ has an UBABS+, in case where $\A=C(K)$).

Suppose now that $\mathbb{P}$ has the c.c.c.

\textbf{Claim 1 :} For every $\alpha<\omega_1$, $D_\alpha = \{p  \in \mathbb{P}: [\alpha, \omega_1)\cap p \neq \emptyset\}$ is a dense set. 

Let $G$ be a generic filter and  $A = \bigcup \{p, p\in G\}$.
For each $\beta \in A$, consider $p_\beta$ such that $\beta\in p_\beta$.
 
Suppose $(p_\alpha)_{\alpha<\omega_1}$ is a $\Delta$-system with root $\Delta$ such that
\begin{enumerate}
\item $\Delta<p_\alpha\setminus \Delta<p_\beta\setminus \Delta$ for every $\alpha<\beta$,
\end{enumerate}

\textbf{Claim 2 :} For every  $\alpha \in A\setminus (p_{\beta}\cap \beta)$, $\alpha\neq \beta$ we have that
$$\tau_\alpha(a_\beta)<\frac{1}{2}.$$

In fact, since $G$ is a filter, there is $q\in G$ such that $q\leq p_{\beta}, p_{\alpha}$. Since $q\leq p_{\beta}$, it follows that $\tau_\xi(a_\beta)< \frac{1}{2}$ for every $\xi \in q\setminus p_{\beta}$.
By hypothesis, $\alpha \in A\setminus (p_{\beta}\cap (\beta))$. Then $\alpha \notin p_{\beta}$ and therefore $\alpha \in q\setminus p_{\beta}$ (because $q\leq p_{\alpha}$) or $\alpha>\beta$. The first case gives us that $\tau_\alpha(a_\beta)< \frac{1}{2}$ and the second case $\tau_\alpha(a_\beta)=0$. In any case, we have that $\tau_\alpha(a_\beta)< \frac{1}{2}$.

By Fodor's lemma, consider an uncountable set $B\subseteq A$ such that $p_{\beta}\cap \beta$ is constantly $F\subseteq A$ for all $\beta \in B$. Then for every $\alpha, \beta \in B\setminus F$ we have that
$$\tau_\alpha(a_\beta)< \frac{1}{2}.$$
This shows us that $(\tau_\alpha, a_\alpha)_{\alpha \in B\setminus F}$ is an UBABS+  on\footnote{In particular, $(\tau_\alpha)_{\alpha \in B\setminus F}$ is an uncountable discrete set in $S(A)$.} $\A$.

\end{proof}

\begin{proposition}Assume MA. Let $\A = C(K)$ be a commutative C*-algebra such that $\A$ has an uncountable positive semibiorthogonal system of pure states. Then $\A$ has an uncountable almost irredundant set.
\end{proposition}
\begin{proof}
It follows from Theorem \ref{mateorema} and Theorem \ref{discretalmost}.
\end{proof}

In the case of noncommutative C*-algebras, Theorem \ref{mateorema} can only guarantee the existence of an uncontable discrete set of states. By Proposition \ref{discrete-positive}, in order to guarantee the existence of an UBABS+ (and therefore, an uncountable almost irredundant set), we need a discrete set of states such that every state in this set can be excised. As proved by Akemann, Anderson and Pedersen \cite{akemann1986excising}, a state can be excised if and only if it belongs to the weak*-closure of the pure states $P(\A)$.

Examples of C*-algebras in which every state can be excised included all C*-algebras which acts irreducible on a Hilbert space $H$ and which contains no non-zero compact operators (Theorem 2 of \cite{glimm1960stone}, see also \cite{brown2005excision}).

To guarantee the existence of an an UBABS+, it would be interesting to obtain an uncountable discrete set of pure states. To this end, we employ a stronger version of Martin's Axiom, namely the Proper Forcing Axiom (PFA). In this article, we will use the following consequence of PFA: There is no S-spaces, i.e., every regular topological space which has an uncountable right-separated sequence has an uncountable left-separated sequence (See Theorem 8.9 of \cite{todorcevic1989partition}). For the statement and other consequences of the Proper Forcing Axiom, see \cite{todorcevic1989partition}.

\begin{theorem}[Proposition 2.5 of \cite{hida2018large}]\label{pfateorema}
Assume PFA. Let $\mathcal{A}$ be a C*-algebra. Suppose $(a_\alpha, \tau_\alpha)_{\alpha<\omega_1}$ is an uncountable positive semibiorthogonal system of pure states. Then there is an uncountable $\Gamma\subset\omega_1$  such that $(\tau_\alpha)_{\alpha \in \Gamma}$ is a discrete set of pure states.
\end{theorem}
\begin{proof}
Since  $(a_\alpha, \tau_\alpha)_{\alpha<\omega_1}$ is an uncountable positive semibiorthogonal system, we have that $Y = (\tau_\alpha)_{\alpha<\omega_1}$ is a right-separated sequence in $P(\A)$. By PFA, $Y$ has an uncountable left-separated sequence. Since a right and left separated sequence is in particular a discrete set, we are done.
\end{proof}

\section{Almost irredundant sets in scattered C*-algebras}\label{scatteredsection}
In this section we prove that assuming the Proper Forcing Axiom, every nonseparable scattered C*-algebra admits an uncountable almost irredundant set. By Theorem \ref{pfateorema} and Theorem \ref{discretalmost}, it is enough to construct an uncountable positive semibiorthogonal system of pure states.

If $K$ is a compact Hausdorff space and there is an uncountable strictly increasing sequence of open sets, then $C(K)$ has an uncountable positive semibiorthogonal system of pure states (Theorem 2 of \cite{lazar1981points}). 
The following lemma can be thought as the noncommutative version of this fact:

\begin{lemma}\label{existence_family_positive}Let $\A$ be a C*-algebra and suppose $(I_\alpha)_{\alpha<\omega_1}$ is an strictly increasing sequence of ideals. Then $\A$ has an uncountable positive semibiorthogonal system of pure states.
\end{lemma}
\begin{proof}
For every $\alpha<\omega_1$, define
$$I_\alpha^\perp = \{\tau\in P(\A): \forall a \in I_\alpha (\tau(a)=0)\}.$$
Then by Theorem 5.4.10 of \cite{murphy1990c}, $(I_\alpha^\perp)_{\alpha<\omega_1}$ is a strictly decreasing sequence of weak*-closed subsets of $P(\A)$. For each $\alpha<\omega_1$, consider $\tau_\alpha \in I_\alpha^\perp\setminus I_{\alpha+1}^\perp$ and a positive element $a_\alpha \in I_{\alpha+1}$ such that $\tau_\alpha(a_\alpha)=1$.
\end{proof}

It follows from Lemma \ref{existence_family_positive} that every C*-algebra with a composition series of size at least $\omega_1$ has an uncountable positive semibiorthogonal system of pure states.

\begin{theorem}\label{PFA-airr-unital}Assume PFA. Let $\A$ be a nonseparable scattered C*-algebra. Then $\A$ has an uncountable almost irredundant set of projections.
\end{theorem}
\begin{proof}Suppose first that $\A$ is unital. Consider the Cantor-Bendixson sequence $(I_\alpha)_{\alpha<\beta}$ of $\A$. If $\beta<\omega_1$, then $\A$ has a quotient with an uncountable irredundant set, and therefore, $\A$ has an uncountable irredundant set of projections(see Proposition 3.16 of \cite{hida2018large}).
Suppose now that $\beta\geq \omega_1$. Then by Lemma \ref{existence_family_positive}, $\A$ has an uncountable positive semibiorthogonal system of pure states. 
To conclude, we apply Theorem \ref{pfateorema} and Theorem \ref{discretalmost}.

Suppose now that $\A$ is not unital. By the first part, if $\tilde{\A}$ is the unitization\footnote{Observe that if $\A$ is scattered, then its unitization is also scattered. See Proposition 2.4 of \cite{ghasemi2018noncommutative}.} of $\A$, then $\tilde{\A}$ has an uncountable almost irredundant set of projections $((a_\alpha, \lambda_\alpha))_{\alpha<\omega_1}$. Since $(a, \lambda)\in \tilde{\A}$ is a projection if and only if $(a, \lambda) = (0,1)$ or $\lambda = 0$ and $a$ is a projection of $\A$, we can assume that $((a_\alpha, \lambda_\alpha))_{\alpha<\omega_1} = ((a_\alpha, 0))_{\alpha<\omega_1}$, where $(a_\alpha)_{\alpha<\omega_1}$ is a family of projections in $\A$ and it is easy to see that $(a_\alpha)_{\alpha<\omega_1}$ is an uncountable almost irredundant set in $\A$.
\end{proof}

\begin{remark}Let $\A$ be a C*-algebra. Consider the following sentences:

\begin{enumerate}
\item $\A$ has an uncountable increasing sequence of ideals.
\item $\A$ has an uncountable positive semibiorthogonal system.
\item $P(\A)$ has an uncountable discrete set.
\item $\A$ has UBABS+.
\item $\A$ has an uncountable almost irredundant set.
\end{enumerate}
We have proved that $1) \Rightarrow 2)$ and $3) \Rightarrow  4) \Rightarrow 5)$ hold in ZFC (Lemma \ref{existence_family_positive} and Proposition \ref{discrete-positive}). Also,  assuming PFA, we have proved that $2) \Rightarrow 3)$ (Theorem \ref{pfateorema}).

Consider $\mathcal{K} = (2^\omega \times \{0,1\}, \tau_{ord})$ the Double arrow space. Then $\mathcal{K}$ is a compact Hausdorff space which is hereditarily separable and hereditarily Lindel\"of. If $\A=C(\mathcal{K})$, then $\A$ has an uncountable (of size $2^\omega$) nice biorthogonal system\footnote{A biorthogonal sytem $(f_\alpha, \tau_\alpha)_{\alpha<\kappa}$ in a Banach space of the form $C(K)$ is nice if for every $\alpha<\kappa$, there are two points $x_\alpha, y_\alpha\in K$ such taht $\tau_\alpha = \delta_{x_\alpha} - \delta_{y_\alpha}$. See \cite{dvzamonja2011ch}.}.

By Theorem 5.4 of \cite{somepiotr11}, since $\A$ has an uncountable nice biorthogonal system, it follows that $\A$ has an uncountable irredundant set (and in particular, an uncountable almost irredundant set).

Since $\mathcal{K}\equiv P(\A)$ is hereditarily separable, it follows that $P(\A)$ does not have an uncountable discrete set. This shows that $5)\not \Rightarrow 3)$.

Also, by Proposition 3.2 of \cite{granero1998convex}, a Banach space of the form $C(K)$ has an uncountable positive semibiorthogonal system if and only if $K$ is not hereditarily Lindel\"of. Since $\mathcal{K}$ is hereditarily Lindel\"of, it follows that $\A$ does not have an uncountable positive semibiorthogonal system. In particular, $5)\not \Rightarrow 2)$ (and therefore $5)\not \Rightarrow 1)$.


\end{remark}


\begin{thebibliography}{10}

\bibitem{akemann1986excising}
Charles~A. Akemann, Joel Anderson, and Gert~K. Pedersen.
\newblock Excising states of {C}*-algebras.
\newblock {\em Canad. J. Math}, 38(5):1239--1260, 1986.

\bibitem{brown2005excision}
Nathanial~P Brown.
\newblock Excision and a theorem of {P}opa.
\newblock {\em Journal of Operator Theory}, pages 3--8, 2005.

\bibitem{kenneth1996c}
Kenneth~R. Davidson.
\newblock {\em C*-algebras by example}, volume~6.
\newblock American Mathematical Soc., 1996.

\bibitem{dixmier1977c}
Jacques Dixmier.
\newblock {\em C*-algebras. Translated from the French by Francis Jellett},
  volume~15.
\newblock North-Holland Math. Library, 1977.

\bibitem{dvzamonja2011ch}
Mirna D{\v{z}}amonja and Istv{\'a}n Juh{\'a}sz.
\newblock Ch, a problem of {R}olewicz and bidiscrete systems.
\newblock {\em Topology and its Applications}, 158(18):2485--2494, 2011.

\bibitem{eilers2018c}
S{\o}ren Eilers and Dorte Olesen.
\newblock {\em C*-algebras and their automorphism groups}.
\newblock Academic press, 2018.

\bibitem{farah2019combinatorial}
Ilijas Farah.
\newblock {\em Combinatorial Set Theory of C*-algebras}.
\newblock Springer, 2019.

\bibitem{ghasemi2018noncommutative}
Saeed Ghasemi and Piotr Koszmider.
\newblock Noncommutative {C}antor-{B}endixson derivatives and scattered
  {C}*-algebras.
\newblock {\em Topology and its Applications}, 240:183--209, 2018.

\bibitem{glimm1960stone}
James Glimm.
\newblock A {S}tone-{W}eierstrass theorem for {C}*-algebras.
\newblock {\em Annals of Mathematics}, pages 216--244, 1960.

\bibitem{granero2003kunen}
Antonio~S Granero, Mar Jim{\'e}nez, Alejandro Montesinos, Jos{\'e}~P Moreno,
  and Anatolij Plichko.
\newblock On the {K}unen--{S}helah properties in {B}anach spaces.
\newblock {\em Studia Mathematica}, 157:97--120, 2003.

\bibitem{granero1998convex}
AS~Granero, M~Jim{\'e}nez Sevilla, and Jos{\'e}~Pedro Moreno.
\newblock Convex sets in {B}anach spaces and a problem of {R}olewicz.
\newblock {\em Studia Math}, 129(1):19--29, 1998.

\bibitem{hida2018large}
Clayton~S. Hida and Piotr Koszmider.
\newblock Large irredundant sets in operator algebras.
\newblock {\em To appear in the Canadian Journal of Mathematics. arXiv preprint
  arXiv:1808.01511}.

\bibitem{handbookconjuntos}
R.~Hodel.
\newblock Cardinal functions.
\newblock {\em Handbook of set-theoretic topology}, I:1--61, 1984.

\bibitem{jensen1978scattered}
Helge~E. Jensen.
\newblock Scattered {C}*-algebras.
\newblock {\em Mathematica Scandinavica}, 41(2):308--314, 1978.

\bibitem{jensen1979scattered}
Helge~E. Jensen.
\newblock Scattered {C}*-algebras {II}.
\newblock {\em Mathematica Scandinavica}, 43(2):308--310, 1979.

\bibitem{somepiotr11}
Piotr Koszmider.
\newblock Some topological invariants and biorthogonal systems in {B}anach
  spaces.
\newblock {\em Extracta Math}, 26(2):271--294, 2011.

\bibitem{kusuda2012c}
Masaharu Kusuda.
\newblock C*-algebras in which every {C}*-subalgebra is {AF}.
\newblock {\em Quarterly journal of mathematics}, 63(3):675--680, 2012.

\bibitem{lazar1981points}
AJ~Lazar et~al.
\newblock Points of support for closed convex sets.
\newblock {\em Illinois Journal of Mathematics}, 25(2):302--305, 1981.

\bibitem{li1992introduction}
Bing-Ren Li.
\newblock {\em Introduction to operator algebras}.
\newblock World Scientific, 1992.

\bibitem{lin1989structure}
Hua~Xin Lin.
\newblock The structure of quasimultipliers of {C}*-algebras.
\newblock {\em Transactions of the American Mathematical Society},
  315(1):147--172, 1989.

\bibitem{murphy1990c}
Gerald~J. Murphy.
\newblock {\em C*-algebras and operator theory}.
\newblock Academic press, 1990.

\bibitem{pedersen1979c}
Gert~K. Pedersen.
\newblock {\em C{*}-algebras and their automorphism groups}.
\newblock Academic press, 1979.

\bibitem{irrstevo93}
Stevo Todorcevic.
\newblock Irredundant sets in {B}oolean algebras.
\newblock {\em Trans. Amer. Math. Soc}, 290:711--723, 1985.

\bibitem{todorcevic1989partition}
Stevo Todorcevic.
\newblock {\em Partition problems in topology}.
\newblock Number~84. American Mathematical Soc., 1989.

\bibitem{quocientstevo}
Stevo Todorcevic.
\newblock Biorthogonal systems and quotient spaces via {B}aire category
  methods.
\newblock {\em Math. Ann.}, 335:687--715, 2006.

\bibitem{tomiyama1963characterization}
Jun Tomiyama.
\newblock A characterization of {C}*-algebras whose conjugate spaces are
  separable.
\newblock {\em Tohoku Mathematical Journal, Second Series}, 15(1):96--102,
  1963.

\bibitem{wojtaszczyk1974linear}
Przemys{\l}aw Wojtaszczyk.
\newblock On linear properties of separable conjugate spaces of {C}*-algebras.
\newblock {\em Studia Mathematica}, 52:143--147, 1974.

\end{thebibliography}

\end{document}